\documentclass[12pt]{iopart}
\usepackage{iopams}
\usepackage{setstack} 
\usepackage{amsthm}
\usepackage{url}
\usepackage{graphicx}

\newtheorem {lemma}{Lemma}

\newtheorem {theorem}{Theorem}

\newtheorem {remark}{Remark}

\newcommand{\R}{{\mathbb R}}

\newcommand{\C}{{\mathbb C}}
\newcommand{\N}{{\mathbb N}}

\def\F{{\mathcal F}}

\def\res{\mathop{\rm res}}
\def\sign{\ensuremath{{\rm sign}}}
\def\linspan{\ensuremath{{\rm span}}}

\newcommand{\supp}{{\rm supp}}
\def \Im{{\rm Im\,}}

\newcommand{\pd}{\partial}
\newcommand{\p}{\partial}

\newcommand{\norm}[1]{\left\|#1 \right\|}

\def\bar{\overline}
\def\hat{\widehat}
\def\tilde{\widetilde}

\renewcommand{\l}{\left}
\renewcommand{\r}{\right}

\begin{document}

\title{Inverse problem for wave equation with sources and observations on disjoint sets}
\author{Matti Lassas and Lauri Oksanen}
\address{Department of Mathematics and Statistics,
University of Helsinki, P.O. Box 68 FI-00014, Finland}
\eads{\mailto{matti.lassas@helsinki.fi}, \mailto{lauri.oksanen@helsinki.fi}}


\begin{abstract}
We consider an inverse problem for a hyperbolic partial differential equation 
on a compact Riemannian manifold. Assuming that 
$\Gamma_1$ and $\Gamma_2$ are  two disjoint open subsets of the boundary of the manifold
we define the restricted Dirichlet-to-Neumann operator $\Lambda_{\Gamma_1,\Gamma_2}$.
This operator corresponds the boundary measurements when we have smooth sources 
supported on $\Gamma_1$ and the fields produced
by these sources are observed on $\Gamma_2$.
We show that when $\Gamma_1$ and $\Gamma_2$ are  disjoint but
their closures intersect at least at one point, then the  restricted Dirichlet-to-Neumann operator $\Lambda_{\Gamma_1,\Gamma_2}$ determines the Riemannian manifold
and the metric on it up to an isometry.
In the Euclidian space, the result yields that 
an anisotropic wave speed inside a compact body is determined,
up to a natural coordinate transformations, by 
measurements on the boundary of the body 
even when wave sources are kept away from receivers. 
Moreover, we show that
if we have three arbitrary non-empty open subsets $\Gamma_1,\Gamma_2$,
and $\Gamma_3$ of the boundary, then the restricted Dirichlet-to-Neumann operators 
$\Lambda_{\Gamma_j,\Gamma_k}$ for $1\leq j<k\leq 3$ determine  
the Riemannian manifold to an isometry. Similar result is proven also for the finite-time
boundary measurements when the hyperbolic equation satisfies an exact
controllability condition.
\end{abstract}
\submitto{Inverse Problems}


\section{Introduction and main results}

Let $M$ be a compact and connected $C^\infty$-smooth manifold of dimension $n$ 
and let $g$ be a $C^\infty$-smooth Riemannian metric on $M$.
Let $q$ be a real-valued $C^\infty$-smooth function on $M$, and
denote by $\Delta_g$ the Laplace-Beltrami operator on $M$.
We consider a hyperbolic inverse problem corresponding to the 2nd order elliptic operator
\begin{equation*}
a(x,D) := -\Delta_g + q(x).
\end{equation*}

In local coordinates $g$ is a positive-definite $C^\infty$-smooth matrix 
$(g_{jk}(x))_{j,k=1}^n$ with the inverse $(g^{jk}(x))_{j,k=1}^n$ and
\begin{equation}
\label{def:Laplace_Beltrami}
a(x,D) u = -|g|^{-1/2} \sum_{j,k = 1}^n \frac{\p}{\p x^j}
 \left( g^{jk} |g|^{1/2} \frac{\p}{\p x^k} u\right) + qu,
\end{equation}
where $|g| := \det (g_{jk})$.
Hence our results cover the setting, where $M \subset \R^n$ is an open domain with smooth 
boundary and $a(x,D)$ is an elliptic operator of the form (\ref{def:Laplace_Beltrami}).

Let $H^s(M)$ be the Sobolev space of $s \in \N$ times weakly differentiable functions on $M$,
and let $H_0^1(M)$ be the $H^1(M)$ closure of $C_0^\infty(M)$, the space of smooth compactly supported functions.
The operator 
\begin{equation} \label{def:dirichlet_op}
Au(x) := a(x,D) u, \quad D(A) := H^2(M) \cap H_0^1(M)
\end{equation}
is self-adjoint in $L^2(M) = L^2(M, dV_g)$, where 
$dV_g$ is the Riemannian volume measure.
In local coordinates $dV_g = |g|^{1/2} dx$.

Denote by $v^f(x,t) = v(x,t)$ the solution of the initial boundary value problem
\begin{eqnarray} \label{def:wave_eq}
& \pd_t^2 v + a(x,D)v = 0\quad {\rm in}\ M \times (0,\infty),
\\& v|_{\pd M \times (0,\infty)} = f, \nonumber
\\& v|_{t=0} = \pd_t v|_{t=0} = 0, \nonumber
\end{eqnarray}
for $ f\in C_0^\infty(\pd M \times (0,\infty))$,  
and define the hyperbolic Dirichlet-to-Neumann operator 
\begin{equation*}
\Lambda : C_0^\infty(\pd M \times (0,\infty)) \to C^\infty(\pd M \times (0,\infty)),
\quad \Lambda f := \pd_\nu v^f|_{\p M \times \R_+},
\end{equation*}
where $\pd_{\nu}$ is the normal derivative on $\pd M$.
In local coordinates the exterior conormal $\nu$ is the covector $(\nu_1, \dots, \nu_n)$ 
with
\begin{equation*}
\sum_{j,k=1}^n \nu_j(x) g^{jk}(x) \nu_k(x) = 1, \quad x \in \p M,
\end{equation*}
and
\begin{equation*}
\pd_\nu = \sum_{j,k=1}^n \nu_j g^{jk} \frac{\p}{\p x^k}.
\end{equation*}

Denote by $\Lambda_{\Gamma_1, \Gamma_2}^T$ 
the restriction of the Dirichlet-to-Neumann operator
\begin{equation*}
\Lambda_{\Gamma_1, \Gamma_2}^T : 
C_0^\infty(\Gamma_1 \times (0,T)) \to C^\infty(\Gamma_2 \times (0,T)),
\end{equation*}
where $\Gamma_1, \Gamma_2 \subset \pd M$ are open.
Furthermore, denote $\Lambda_{\Gamma_1, \Gamma_2} := \Lambda_{\Gamma_1, \Gamma_2}^\infty$.

It is well known, that the map $\Lambda$
determine the manifold $(M, g)$ up to an isometry \cite{BeKu}.
This is also true for the restriction $\Lambda_{\Gamma, \Gamma}^T$
when $\Gamma$ is nonempty and $T$ is sufficiently large \cite{KaKu}.



In many applications observations of physical fields can not be
done on the same locations where the sources of the fields
are. For instance, in imaging in Earth Sciences, elastic
or acoustic fields are often implemented using explosions \cite{Sy, Ra}.
In such a case observation devices need to be far away from the sources.

Similarly, in electromagnetic imaging, it is technically difficult
to use electrodes at the same time as sources and for
making observations. These are typical examples of cases
where the observation devices and the sources
of the fields are supported on disjoint sets.

In this paper we show, that for certain collections of pairs $(\Gamma_1, \Gamma_2)$ of
open and disjoint subsets of $\pd M$, the operators $\Lambda_{\Gamma_1, \Gamma_2}^T$ 
determine the manifold $(M, g)$ up to an isometry.

\begin{theorem} \label{thm:touching_main}
Let $\Gamma_1, \Gamma_2, \Sigma \subset \p M$ be open sets such that 
$\bar \Gamma_1, \bar \Gamma_2 \subset \Sigma$ and
$\overline \Gamma_1 \cap \overline \Gamma_2 \ne \emptyset$. 
Then $\Sigma$, given as a smooth manifold, and the operator $\Lambda_{\Gamma_1, \Gamma_2}$
determine the manifold $(M,g)$ up to an isometry.
\end{theorem}

\begin{theorem} 
\label{thm:combining_data_infinite_time}
Let $\Gamma_1$, $\Gamma_2, \Gamma_3 \subset \p M$ be open and nonempty. Then 
the smooth manifolds $\Gamma_p$, $p = 1, 2, 3$, and
the operators
\begin{equation*}
\Lambda_{\Gamma_1, \Gamma_2},\  
\Lambda_{\Gamma_1, \Gamma_3},\  
\Lambda_{\Gamma_2, \Gamma_3}  
\end{equation*}
determine the manifold $(M,g)$ up to an isometry.
\end{theorem}

For measurements on a finite time interval, we prove a 
theorem similar to Theorem \ref{thm:combining_data_infinite_time} 
under an additional controllability assumption:
\begin{itemize}
\item[(A)] For any $w \in L^2(M)$ there is 
a boundary value $f \in L^2(\pd M \times (0,\infty))$ satisfying
\begin{equation*}
v^f(T/2) = w, \quad \supp(f) \subset \Gamma_3 \times (0,T),
\end{equation*}
where $v^f$ is the solution of the equation 
(\ref{def:wave_eq}) and $\Gamma_3 \subset \p M$.
\end{itemize}

Let us comment the controllability assumption (A) when $M$ is embedded in $\R^n$.
Property (A) follows from the geometric control condition of 
Bardos, Lebeau and Rauch \cite{BLR}, 
which yields exact controllability of the wave equation.
Property (A) follows also 
from existence of a strictly convex function $h$ on $\bar M$
with respect to the Riemannian metric $g$. 

Suppose that $h \in C^2(\bar M)$ is strictly convex and that 
$\rho > 0$ is a lower bound for the Hessian of $h$ 
in the Riemannian metric $g$, that is
\begin{equation*}
D^2 h(X, X) \ge \rho |X|_g, \quad X \in T_x M,\ x \in M.
\end{equation*}
By \cite{LTY}, (A) holds if 
\begin{equation*}
T > \frac{4}{\rho} \max_{x \in \bar M} |\nabla_g h|_g, \quad
\sup_{x \in \Gamma_3} \nabla_g h(x) \cdot \bar \nu(x) \le 0,
\end{equation*}
where $\nabla_g$ and $|\cdot|_g$ are the gradient and length 
with respect to the Riemannian metric $g$,
$\bar \nu$ is the Euclidean unit outward normal to $\p M \subset \R^n$ 
and $\nabla_g h \cdot \bar \nu$ is the Euclidean inner product.
We refer to \cite{LTY} for examples of Riemannian manifolds $(M, g)$ 
having a strictly convex function $h$.

%

\begin{theorem}
\label{thm:combining_data_finite_time}
Let $\Gamma_1, \Gamma_2, \Gamma_3 \subset \p M$ be open and nonempty.
If the controllability assumption (A) holds for 
$\Gamma_3$ and $T > 0$, and 
\begin{equation*}
4 d(x,y) < T, \quad x \in \Gamma_p,\ y \in M,\ p = 1,2,3,
\end{equation*}
then the Riemannian manifolds $(\Gamma_p, g|_{\Gamma_p})$, $p = 1,2,3$, and the operators
\begin{equation*}
\Lambda_{\Gamma_1, \Gamma_2}^T,\  
\Lambda_{\Gamma_1, \Gamma_3}^T,\  
\Lambda_{\Gamma_2, \Gamma_3}^T  
\end{equation*}
determine the manifold $(M, g)$ up to an isometry.
\end{theorem}

The proofs of these theorems consist of showing that the data 
determine, up to a gauge transformation, the boundary spectral data 
of the operator $A$ on a part of the boundary.
The manifold is then determined up to an isometry,
as can be seen using the boundary control method \cite{Be1, BeKu, KaKu, KKL}.

Notice also, that the operator $\Lambda_{\Gamma_1, \Gamma_2}^T$ determines the operator
$\Lambda_{\Gamma_2, \Gamma_1}^T$ by a time reversal argument.

\begin{figure}
\begin{center}
\includegraphics[width=0.4\textwidth]{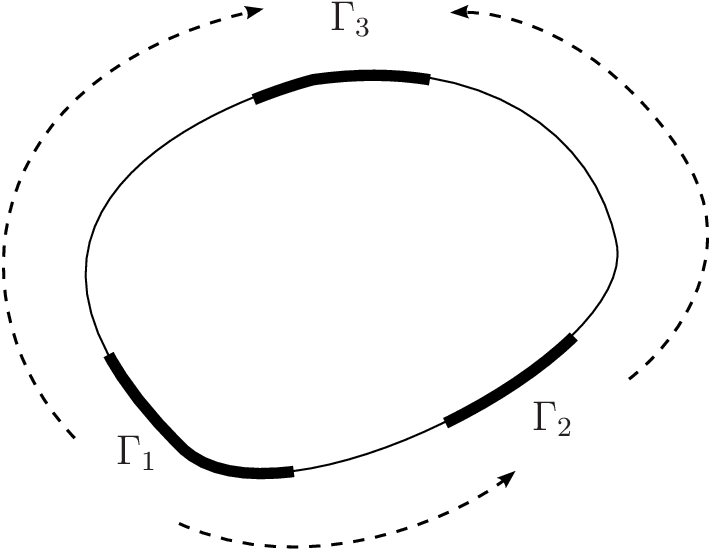}
\hspace{0.05\textwidth}
\includegraphics[width=0.4\textwidth]{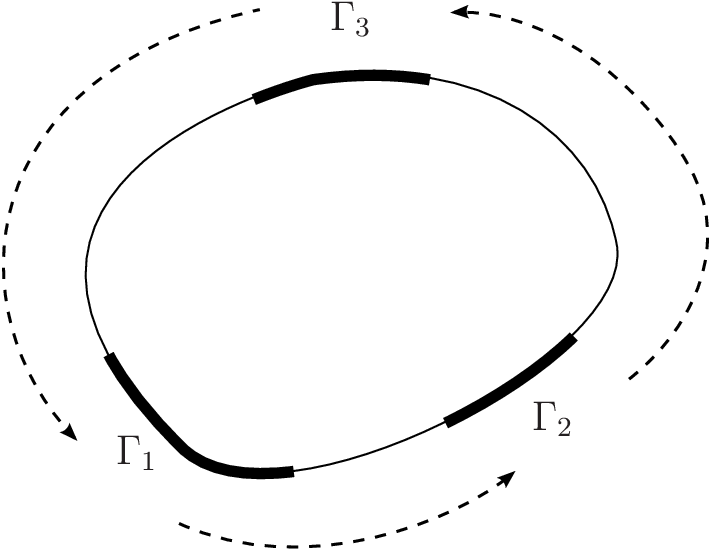}
\end{center}
\caption{On left, the measurements 
$\Lambda_{\Gamma_1, \Gamma_2},\  
\Lambda_{\Gamma_1, \Gamma_3},\  
\Lambda_{\Gamma_2, \Gamma_3}$ 
are shown as arrows pointing from the support of sources to the support of observations.
Using Lemma \ref{lem:symmetrization_of_data} we can change the direction of 
any arrow in the picture on left. Hence also the measurements shown on right 
are covered by Theorems \ref{thm:combining_data_infinite_time}
and \ref{thm:combining_data_finite_time}.
}
\end{figure}

\begin{lemma}
\label{lem:symmetrization_of_data}
Let $\Gamma_1, \Gamma_2 \subset \p M$ be open and nonempty.
Let $T > 0$, and define the time reversal operator $Rf(x,t) := f(x, T-t)$.
If $f \in C_0^\infty(\Gamma_1 \times (0,T))$ and $h \in C_0^\infty(\Gamma_2 \times (0,T))$, then
\begin{equation*}
(f, \Lambda_{\Gamma_2, \Gamma_1}^T h)_{L^2(\p M \times (0,T))} 
= (R \Lambda_{\Gamma_1, \Gamma_2}^T R f, h)_{L^2(\p M \times (0,T))}.
\end{equation*}
Hence $(\Gamma_j, g|_{\Gamma_j})$, $j = 1, 2$, given as Riemannian manifolds, 
and the operator $\Lambda_{\Gamma_1, \Gamma_2}^T$ determine the operator
$\Lambda_{\Gamma_2, \Gamma_1}^T = (R \Lambda_{\Gamma_1, \Gamma_2}^T R)^t$.
\end{lemma}

This result is relatively well known, see e.g. \cite{BKLS, DKL},
but for the sake of completeness, we will give a proof in the appendix.

Let us review previous results on the topic.
The inverse problem for isotropic wave equation on a compact manifold with measurements
on the whole boundary was solved by Belishev and Kurylev \cite{BeKu}.
This was based on the boundary control method originally developed in \cite{Be1} 
for wave equation on a bounded domain of $\R^n$. 
The inverse problems for more general hyperbolic equations on a compact Riemannian manifold
with sources and observations on the same open subset $\Gamma$ of the boundary
has been studied by Katchalov and Kurylev \cite{KaKu}, see also \cite{KuLa}.
Similar problem has recently been studied for non-compact manifolds in \cite{IsKuLa,KaKuLa_conf}.

Inverse problems for elliptic equations with data on a part of the boundary have
been studied intensively as they are the natural generalization of the Calder\'on's
inverse problem for the conductivity equation \cite{Calderon}.

When measurements are given on the whole boundary, 
the inverse problem for Schr\"odinger equation on a bounded domain of $\R^n$, $n\geq 3$,
and hence for isotropic conductivity equation,
was solved by Sylvester and Uhlmann in \cite{SU}.
The corresponding two dimensional problem for isotropic conductivity equation 
was solved first by Nachman in \cite{Na} for $C^2$ conductivities, and
for $L^\infty$ conductivities, for
which Calder\'on's inverse problem was originally posed,
by Astala and P\"aiv\"arinta in \cite{AP}.
Recently, also the inverse problem for Schr\"odinger equation
on a bounded domain of dimension two with measurements on the whole boundary 
was solved by Bukgheim in \cite{Bu}.
The corresponding problem on a compact Riemannian surface was later solved in \cite{GT2}.

The inverse problem for Schr\"odinger equation
on a bounded domain of $\R^n$, $n\geq 3$, with observations on an open subset 
$\Gamma$ of the boundary was solved in \cite{KSU}.
The inverse problem for Schr\"odinger equation
on a bounded domain of $\R^n$, $n = 2$, with sources and observations on the same open subset 
$\Gamma$ of the boundary was solved by Imanuvilov, Uhlmann and Yamamoto in \cite{IUY}.
The corresponding problem on a compact Riemannian surface was later solved in \cite{GT}.
For related results with measurements on a part of the boundary
, see \cite{BuU,Greenleaf,Isakov}.

The inverse problem for the Laplace-Beltrami operator $\Delta_g$ on a compact Riemannian manifold
with sources and observations on the same open subset 
$\Gamma$ of the boundary 
has been studied on analytic Riemannian manifolds of dimension $n \ge 3$
in \cite{LeU,LTU}, and on Riemannian surfaces in \cite{LU}, see also
\cite{Khenkin1,Khenkin2}.
The inverse problem for the Laplace-Beltrami operator in 
dimensions $n \ge 3$ is open in general, even when measurements are given on the whole boundary.
For positive results under certain geometrical conditions see \cite{DsfKSU}.

\section{Spectral analysis of the data}

Denote by $(\lambda_j)_{j \in \N}$ the increasing sequence of distinct eigenvalues 
of the operator $A$
and let $(\phi_k)_{k \in \N}$ be an orthonormal basis of real-valued 
$C^\infty$-smooth eigenfunctions.
Moreover, let $(I_j)_{j \in \N}$ be a partition of $\N$ such that 
$(\phi_k)_{k \in I_j}$ is a basis for the space of eigenfunctions
corresponding $\lambda_j$.

Let $f \in C_c^\infty(\p M \times (0, \infty))$, and
consider $\Lambda f$ also as a function in $C^\infty(\p M \times \R)$ by 
defining $\Lambda f(\cdot,t) = 0$ for $t \le 0$.
There is a constant $C > 0$ such that for $x \in \p M$, 
the Fourier transform $\F_{t \to \tau} \Lambda f(x)$ 
is an analytic function of $\tau$ when $\Im \tau < -C$.
It is known (see e.g. \cite{KKL}), that 
$\F_{t \to \tau} \Lambda f(x)$ extends to a meromorphic function of $\tau \in \C$,
and that it may have poles only at points $\sqrt{\lambda_j}$.
Moreover, the residues at these points are
\begin{equation*}
\res_{\tau = \sqrt{\lambda_j}}
 \F_{t \to \tau} \Lambda f (x)
= \sum_{k \in I_j} \l(\hat f(\cdot, \sqrt{\lambda_j}), \p_\nu \varphi_k \r)_{L^2(\p M, dS_g)} 
    \p_\nu \varphi_k (x),
\end{equation*}
where $\hat f(x,\tau) = (\F_{t \to \tau}f)(x, \tau)$ and $dS_g$ is the Riemannian surface measure.

If $j \in \N$, $a_k$ are constants for $k \in I_j$, and the linear combination
\begin{equation*}
\sum_{k \in I_j} a_k \pd_\nu \phi_k = \pd_\nu \l( \sum_{k \in I_j} a_k \phi_k \r)
\end{equation*}
vanish on a nonempty open subset of $\pd M$, then
$a_k = 0$ for all $k \in I_j$ by unique continuation,
see e.g. \cite{Leis}.
Hence for an open nonempty set $\Gamma \subset \p M$ and $j \in \N$, 
the functions $(\pd_\nu \phi_k|_{\Gamma})_{k \in I_j}$ are linearly independent.

Let $\Gamma_1, \Gamma_2 \subset \p M$ be open and nonempty.
By linear independence 
and smoothness of $\p_\nu \phi_k$,
there are $f \in C_0^\infty(\Gamma_1)$ and $x \in \Gamma_2$
such that 
\begin{equation*}
\sum_{k \in I_j} (f, \p_\nu \varphi_k)_{L^2(\p M, dS_g)} \p_\nu \varphi_k (x)
\ne 0.
\end{equation*}
Moreover, for fixed $\tau \in \C$, the map $f \mapsto \hat f(\cdot, \tau)$ from 
$C_0^\infty(\Gamma_1 \times (0,\infty))$ to $C_0^\infty(\Gamma_1)$
is surjective. 

Hence the operator $\Lambda_{\Gamma_1, \Gamma_2}$
determines the eigenvalues $\lambda_j$
and the operators
\begin{eqnarray}
\label{def:mixed_eigenfunctions_operator}
L_{\Gamma_1, \Gamma_2; j} &: C_c^\infty(\Gamma_1 \times (0, \infty)) \to C^\infty(\Gamma_2 \times (0, \infty)),
\\\nonumber L_{\Gamma_1, \Gamma_2; j} f &:= \sum_{k \in I_j} (f, \p_\nu \varphi_k)_{L^2(\p M, dS_g)} \p_\nu \varphi_k|_{\Gamma_2}.
\end{eqnarray}


\section{Inverse problem with disjoint sources and observations}

In this section we prove Theorem \ref{thm:touching_main}.

\begin{lemma}
\label{lem:symmetry_for_derivatives_1d}
Suppose that $f,h \in C^\infty(\R)$ are such that
\begin{equation*}
\pd_x^{j} \pd_y^{k} (h(x) f(x) f(y))|_{x=0, y=0}
=
\pd_x^{k} \pd_y^{j} (h(x) f(x) f(y))|_{x=0, y=0}
\end{equation*}
for all $j,k \in \N$.
Then $\pd^j f(0) = 0$ for all $j \in \N := \{0, 1, 2, \dots\}$ or 
$\pd^k h(0) = 0$ for all positive $k \in \N$.
\end{lemma}
\begin{proof}
Assume that the claim is not valid.
Then there exist $j \in \N$ and $k \in \N \setminus \{0\}$
such that $\pd^j f(0) \ne 0$ and $\pd^k h(0) \ne 0$.
Let us next consider the smallest such integers $j$ and $k$.

By Leibniz's formula
\begin{eqnarray*}
0 &= 
 \pd_x^{j+k} \pd_y^{j} (h(x) f(x) f(y))|_{x=0, y=0}
 - \pd_x^{j} \pd_y^{j+k} (h(x) f(x) f(y))|_{x=0, y=0}
\\&=
 \sum_{l=0}^{j+k} {j+k \choose l} \pd^l h(0) \pd^{j+k-l} f(0) \pd^j f(0) 
\\&\quad 
 - \sum_{m=0}^{j} {j \choose m} \pd^m h(0) \pd^{j-m} f(0) \pd^{j+k} f(0) 
\\&=
 S_1 
 + {j+k \choose k} \pd^k h(0) \pd^{j} f(0) \pd^j f(0)
 + S_2 - S_3,
\end{eqnarray*}
where 
\begin{eqnarray*}
S_1 &:= \sum_{l=1}^{k-1} {j+k \choose l} \pd^l h(0) \pd^{j+k-l} f(0) \pd^j f(0),
\\
S_2 &:= \sum_{l=k+1}^{j+k} {j+k \choose l} \pd^l h(0) \pd^{j+k-l} f(0) \pd^j f(0),
\\
S_3 &:= \sum_{m=1}^{j} {j \choose m} \pd^m h(0) \pd^{j-m} f(0) \pd^{j+k} f(0),
\end{eqnarray*}
and the terms with indices $l=0$ and $m=0$ have cancelled each other out.

As $k$ is the smallest positive integer such that $\pd^k h(0) \ne 0$,
we have $\pd^l h(0) = 0$ in the sum $S_1$, and so $S_1 = 0$.
As $j$ is the smallest integer such that $\pd^j f(0) \ne 0$,
we have $\pd^{j+k-l} f(0) = 0$ in the sum $S_2$ and
$\pd^{j-m} f(0) = 0$ in the sum $S_3$, thus $S_2 = S_3 = 0$.

Hence $\pd^k h(0) (\pd^{j} f(0))^2 = 0$, which is a contradiction
with the assumption that $\pd^j f(0) \ne 0$ and $\pd^k h(0) \ne 0$.
This proves the claim.
\end{proof}

In the proof of the next lemma we use the equation
\begin{equation}
\label{eqn:directed_derivative}
\pd_t^j f(tv) = \sum_{|\alpha| = j} \frac{j!}{\alpha!} \pd^\alpha f(tv) v^\alpha,
\end{equation}
where $f \in C^\infty(\R^n)$, $t \in \R$, $v \in \R^n$ and $j \in \N$,

\begin{lemma}
\label{lem:symmetry_for_derivatives}
Suppose that $f,h \in C^\infty(\R^n)$ are such that
\begin{equation*}
\pd_x^{\alpha} \pd_y^{\beta} (h(x) f(x) f(y))|_{x=0, y=0}
=
\pd_x^{\beta} \pd_y^{\alpha} (h(x) f(x) f(y))|_{x=0, y=0}
\end{equation*}
for all multi-indices $\alpha, \beta \in \N^n$.
Then $\pd^\alpha f(0) = 0$ for all multi-indices $\alpha \in \N^n$ or 
$\pd^\beta h(0) = 0$ for all nonzero multi-indices $\beta \in \N^n$.
\end{lemma}
\begin{proof}
Let $j,k \in \N$ and $v \in \R^n$. By (\ref{eqn:directed_derivative})
\begin{eqnarray*}
&\pd_t^{j} \pd_s^{k} (h(tv) f(tv) f(sv))|_{t=0, s=0}
- \pd_t^{k} \pd_s^{j} (h(tv) f(tv) f(sv))|_{t=0, s=0}
\\&\quad= 
 \sum_{|\alpha|=j} \sum_{|\beta|=k} \frac{j!}{\alpha!} \frac{k!}{\beta!} v^\alpha v^\beta
   \left( \pd^\alpha (hf)(0) \pd^\beta f(0) - \pd^\beta (hf)(0) \pd^\alpha f(0) \right) = 0.
\end{eqnarray*}
Hence 
\begin{equation*}
\pd_t^{j} \pd_s^{k} (h(tv) f(tv) f(sv))|_{t=0, s=0}
=
\pd_t^{k} \pd_s^{j} (h(tv) f(tv) f(sv))|_{t=0, s=0}
\end{equation*}
for all $j,k \in \N$ and $v \in \R^n$.

Define the sets
\begin{eqnarray*}
F &:= \{ v \in \R^n : \pd_t^j f(tv)|_{t=0} = 0 \text{ for all $j \in \N$} \},
\\
H &:= \{ v \in \R^n : \pd_t^k h(tv)|_{t=0} = 0 \text{ for all positive $k \in \N$} \}.
\end{eqnarray*}
The sets $F$ and $H$ are closed by smoothness of $f$ and $h$, respectively.
Lemma \ref{lem:symmetry_for_derivatives_1d} gives that $F \cup H = \R^n$.
If $F \ne \R^n$, then $\R^n \setminus F$ is open, nonempty and contained in $H$.
Thus $F$ or $H$ contains an open nonempty subset.

Suppose that $U \subset F$ is open and nonempty.
Let $j \in \N$, and define the polynomial 
\begin{equation*}
p(v) := \sum_{|\alpha|=j} \frac{j!}{\alpha!} \pd^\alpha f(0) v^\alpha.
\end{equation*}
By (\ref{eqn:directed_derivative}), $p(v) = \pd_t^j f(tv)|_{t = 0}$,
and $p$ vanish in $U$.
Using unique continuation for real analytic functions we see that $p = 0$ in $\R^n$, and so 
the coefficients of $p$ vanish.
As $j$ can be chosen freely, $\pd^\alpha f(0) = 0$ for all multi-indices $\alpha$.

Similarly, if there exists an open and nonempty $V \subset H$, then 
$\pd^{\alpha} h(0) = 0$ for all nonzero multi-indices $\alpha$.
\end{proof}

\begin{remark}
\label{rem:derivatives_in_local_coordinates}
Let $U$ be a $C^\infty$-smooth manifold of dimension $n$,
$f \in C^\infty(U)$ and $p \in U$.
If $\pd^\alpha f(0) = 0$ for all multi-indices $\alpha \in \N^n$
in some local coordinates taking $p$ to $0$,
then $\pd^\alpha f(0) = 0$ for all multi-indices $\alpha \in \N^n$
in all local coordinates taking $p$ to $0$.
\end{remark}
%

\begin{lemma} 
\label{lem:unique_continuation}
Let $\phi$ be an eigenfunction of the operator A 
corresponding to an eigenvalue $\lambda$, and let $p_0 \in \pd M$.
Then in any local coordinates of $\p M$ taking $p_0$ to $0$,
there is a multi-index $\alpha \in \N^{n-1}$ such that
$\pd^\alpha \pd_\nu \phi(0) \ne 0$.
\end{lemma}
\begin{proof}
Assume that the claim is not valid. Then $\pd^\alpha \pd_\nu \phi(0) = 0$
for all $\alpha \in \N^{n-1}$ in some local coordinates of $\p M$ taking $p_0$ to $0$.

Consider boundary normal coordinates of $\bar M$ taking $p_0$ to $0$.
We may suppose that the coordinates map a small neigborhood $V$ of $p_0$
onto $B(0, \epsilon) \times [0, \epsilon)$,
where $B(0, \epsilon) \subset \R^{n-1}$ is a ball of radius $\epsilon > 0$ centered at the origin.
Then these coordinates take a boundary point $p' \in \p M \cap V$ 
to a point $(x', 0) \in B(0, \epsilon) \times \{0\}$, where $x' = (x^1, \dots, x^{n-1})$.

The special property of the boundary normal coordinates is, that 
a point $p \in \bar M \cap V$ has coordinates
$(x', x^n) \in B(0, \epsilon) \times [0, \epsilon)$, where
$x^n = d(p, \p M)$ and $x'$ are the coordinates of the 
unique boundary point $p' \in \p M$ such that $d(p, p') = d(p, \p M)$.

Moreover, in the coordinates $(x', x^n)$ the equation 
\begin{equation*}
(- \Delta_g + q) \phi = \lambda \phi,
\end{equation*}
has the form
\begin{equation}
\label{eq:eigeneq_in_boundary_normal_coordinates}
-\pd_{x^n}^2 \phi - \sum_{j,k=1}^{n-1} g^{jk} \pd_{x^j} \pd_{x^k} \phi 
  + \sum_{j=1}^n a^j \pd_{x^j} \phi + a^0 \phi = \lambda \phi,
\end{equation}
for some $a^j \in C^\infty(B(0,\varepsilon) \times [0, \varepsilon))$, $j = 0, \dots, n$,
see e.g. \cite{chavel}.

Let us show, that $\phi = 0$.
Let $b \in \N$ and $\alpha \in \N^{n-1}$. 
By applying the operator $\pd_{x'}^\alpha \pd_{x^n}^b$ on the both sides of 
(\ref{eq:eigeneq_in_boundary_normal_coordinates}), we get
\begin{equation*}
\pd_{x'}^\alpha \pd_{x^n}^{b+2} \phi(0) = \pd_{x'}^\alpha \pd_{x^n}^b (
- \sum_{j,k=1}^{n-1} g^{jk} \pd_{x^j} \pd_{x^k} \phi 
  + \sum_{j=1}^n a^j \pd_{x^j} \phi + a^0 \phi - \lambda \phi)|_{x=0}.
\end{equation*}
The right hand side of this equation is a linear combination
of functions 
\begin{equation*}
\pd_{x'}^{\alpha'} \pd_{x^n}^{b'} \phi, \quad \alpha' \in \N, b' \le b + 1
\end{equation*}
at the point $x = 0$. 

The equations $\pd_{x'}^{\alpha} \phi(0) = 0$ hold for all $\alpha \in \N^{n-1}$ 
by the boundary condition $\phi|_{\p M} = 0$.
Furthermore, we have by Remark \ref{rem:derivatives_in_local_coordinates},
that $\pd_{x'}^\alpha \pd_{x^n} \phi(0) = 0$ for all $\alpha \in \N^{n-1}$.
Using induction we see that 
$\pd_{x'}^\alpha \pd_{x^n}^b \phi(0) = 0$ for all $\alpha \in \N^{n-1}$ and $b \in \N$.

Define the odd and even reflection operators 
\begin{equation*}
R_o f(x',x^n) := (\sign\ x^n) f(x',|x^n|), 
\quad R_e f(x',x^n) := f(x',|x^n|),
\end{equation*}
where $f \in C(B(0,\varepsilon) \times [0, \varepsilon))$, $x' \in B(0,\varepsilon)$
and $x^n \in (-\varepsilon, \varepsilon)$. 

Also, define $\tilde \phi := R_o \phi$, $\tilde g^{jk} := R_e g^{jk}$, 
$\tilde a^n := R_o a^n$ and $\tilde a^j := R_e a^j$ for $j = 0, \dots, n-1$.
Denote $U := B(0,\varepsilon) \times (-\varepsilon, \varepsilon)$.

As $\phi|_{x^n=0} = 0$, we see that $\tilde \phi \in H^2(U)$ and 
\begin{equation*}
\pd_{x'}^\alpha \pd_{x^n}^b \tilde \phi(x',x^n) = (\sign\ x^n)^{b+1} \pd_{x'}^\alpha \pd_{x^n}^b \phi(x',|x^n|), 
\quad |\alpha| + b \le 2.
\end{equation*}
Moreover, $\tilde g^{jk}$ is Lipschitz continuous in $U$,
$a^j \in L^\infty(U)$ for $j = 0, \dots, n$, and 
\begin{equation*}
-\pd_{x^n}^2 \tilde \phi - \sum_{j,k=1}^{n-1} \tilde g^{jk} \pd_{x^j} \pd_{x^k} \tilde \phi 
  + \sum_{j=1}^n \tilde a^j \pd_{x^j} \tilde \phi + \tilde a^0 \tilde \phi = \lambda \tilde \phi,
\end{equation*}
where the both sides are considered as functions in $L^2(U)$.
Hence for some constant $C > 0$
\begin{equation*}
|\pd_{x^n}^2 \tilde \phi + \sum_{j,k=1}^{n-1} \tilde g^{jk} \pd_{x^j} \pd_{x^k} \tilde \phi| 
\le C \sum_{|\alpha| + b \le 1} |\pd_{x'}^\alpha \pd_{x^n}^b \tilde \phi|,
\quad \text{in $U$}. 
\end{equation*}

Since $\phi \in C^\infty(B(0,\varepsilon) \times [0, \varepsilon))$ 
vanishes up to arbitrary degree in origin, Taylor's formula gives
for any $m \in \N$ a constant $C_m > 0$ such that
\begin{equation*}
\int_{B(0,r)} \int_0^r |\phi(x)|^2 dx^n dx' \le \int_{B(0,r)} \int_0^r C_m r^m  dx^n dx',
\quad \text{ as $r \to 0$.}
\end{equation*}
Hence for any $m \in \N$, there is a constant $C_m' > 0$ such that
\begin{equation*}
\int_{B(0,r)} \int_{-r}^r  |\tilde \phi(x)|^2 dx^n dx' \le C_m' r^m, 
\quad \text{ as $r \to 0$.}
\end{equation*}

By H\"ormander's strong unique continuation result \cite{hormander:uc} this yields, 
that $\tilde \phi = 0$ in $U$.
In particular, $\phi = 0$ around some point $q \in M$.
As $M$ is connected, unique continuation gives
that $\phi = 0$ in $M$. This is a contradiction with the assumption that
$\phi$ is an eigenfunction, and the claim is proved.
\end{proof}

\begin{remark}
\label{lem:spans}
Let $X$ and $Y$ be Hilbert spaces, and 
let $u_1, \dots, u_N \in X$ and $v_1, \dots, v_N \in Y$ be linearly independent. 
Suppose that $D \subset X$ is a dense subspace, and define
\begin{equation*}
L : D \to Y, \quad Lf := \sum_{k=1}^N (f, u_k)_X v_k.
\end{equation*}
Then $L$ determines the unique bounded extension $\tilde L : X \to Y$,
and its adjoint $\tilde L^* : Y \to X$.
Hence, $L$ determines the spaces
\begin{equation*}
\linspan(v_1, \dots, v_N) = \tilde L(X), \quad \linspan(u_1, \dots, u_N) = \tilde L^*(Y).
\end{equation*}
\end{remark}
 
%
%

\begin{theorem} 
\label{thm:boudary_data_to_spectral_data}
Let $\Gamma_1, \Gamma_2, \Sigma \subset \p M$ be open, 
$\bar \Gamma_1, \bar \Gamma_2 \subset \Sigma$, and
$\overline \Gamma_1 \cap \overline \Gamma_2 \ne \emptyset$. 
Then the smooth manifold $\Sigma$ and the collection
\begin{equation}
\label{eqn:data_of_thm_boudary_data_to_spectral_data}
\{ (\lambda_j, L_{\Gamma_1, \Gamma_2; j})\ |\ j \in \N \}
\end{equation}
determine boundary spectral data up to a constant gauge transformation on $\Gamma_2$.
That is, one can find a collection
\begin{equation}
\label{eq:boundary_spectral_data}
\{ (\lambda_j, (\pd_\nu \psi_k|_{\Gamma_2})_{k \in I_j})\ |\ j \in \N \},
\end{equation}
where for an unknown constant $C > 0$ not depending on $j$ or $k$,
$(C\psi_k)_{k \in I_j}$ is an orthonormal basis of eigenfunctions in $L^2(M)$
corresponding the eigenvalue $\lambda_j$.
\end{theorem}
\begin{proof}
Choose a smooth positive measure $d\mu$ on $\Sigma$. 
Then there is a positive function $\eta \in C^\infty(\Sigma)$ such that $\eta d\mu = dS_g|_\Sigma$.
As $\eta > 0$, the functions $(\eta \pd_\nu \phi_k|_{\Gamma_1})_{k \in I_j}$
are linearly independent for all $j \in \N$.

For all $j \in \N$, denote $L_j := L_{\Gamma_1, \Gamma_2; j}$, and define
\begin{eqnarray*}
&l_j(x,y) := \sum_{k \in I_j} \eta(x) \pd_\nu \phi_k(x) \pd_\nu \phi_k(y), 
  \quad x,y \in \pd M,
\\&\tilde E_j^1 := \linspan(\eta \pd_\nu \phi_k|_{\Gamma_1})_{k \in I_j},
\quad E_j^2 := \linspan(\pd_\nu \phi_k|_{\Gamma_2})_{k \in I_j}.
\end{eqnarray*}

Note, that for the smallest eigenvalue $\lambda_0$, the space of eigenfunctions is one dimensional 
(see e.g. \cite[Thm 6.5.2]{evans}),
and so 
\begin{equation*}
l_0(x,y)
= \eta(x) \pd_\nu \phi_0(x) \pd_\nu \phi_0(y).
\end{equation*}

Consider a positive function $\tilde \eta \in C^\infty(\Sigma)$ and 
real-valued functions $e_k \in C^\infty(\Sigma)$, $k \in \N$, such that
the following three conditions hold:
\begin{itemize}
\item[(A1)] If $x_0 \in \bar \Gamma_1 \cap \bar \Gamma_2$, then
in local coordinates of $\Sigma$ taking $x_0$ to $0$
\begin{equation*}
\pd_x^\alpha \pd_y^\beta \left( \frac{l_0(x,y)}{\tilde \eta(x)} \right)|_{x = 0, y = 0} 
= \pd_x^\beta \pd_y^\alpha \left( \frac{l_0(x,y)}{\tilde \eta(x)} \right)|_{x = 0, y = 0} 
\end{equation*}
for all multi-indices $\alpha, \beta \in \N^{n-1}$.
\item[(A2)] $\linspan(\tilde \eta e_k|_{\Gamma_1})_{k \in I_j} = \tilde E_j^1$ and
$\linspan(e_k|_{\Gamma_2})_{k \in I_j} = E_j^2$ for all $j \in \N$. 
\item[(A3)] $\tilde L_j = L_j$ for all $j  \in \N$, where
\begin{equation*}
\tilde L_j f(y) := \sum_{k \in I_j} (f, \tilde \eta e_k)_{L^2(\pd M, d\mu)} e_k(y),
\quad f \in C_0^\infty(\Gamma_1),\ y \in \Gamma_2.
\end{equation*}
\end{itemize}

Such functions $\tilde \eta$ and $e_k$ exist. 
For example, $\tilde \eta = \eta$ and $e_k = \pd_\nu \phi_k|_\Sigma$ satisfy the conditions.

Next we show the following two statements.
\begin{itemize}
\item[(i)] We can verify using the data (\ref{eqn:data_of_thm_boudary_data_to_spectral_data}),
whether any given functions 
$\tilde \eta \in C^\infty(\Sigma)$ and $e_k \in C^\infty(\Sigma)$, $k \in \N$,
satisfy the conditions (A1), (A2) and (A3).
\item[(ii)] There is an orthonormal basis of eigenfunctions $(\psi_k)_{k \in \N}$ 
of operator $A$ and a constant $C > 0$, not depending on $k$, such that
\begin{equation*}
e_k|_{\Gamma_2} = C \pd_\nu \psi_k|_{\Gamma_2}, \quad k \in \N.
\end{equation*}
\end{itemize}

If (i) holds, then the data (\ref{eqn:data_of_thm_boudary_data_to_spectral_data}) determine 
the nonempty collection 
\begin{equation*}
\{ (e_k|_{\Gamma_2})_{k \in \N}\ |\ \text{(A1)-(A3) hold 
with a positive $\tilde \eta \in C^\infty(\Sigma)$} \},
\end{equation*}
and if (ii) holds, then any element from this collection determines a
collection of type (\ref{eq:boundary_spectral_data}).
So the claim of the theorem is proved after proving (i) and (ii).

Let us show the claim (i). 
Clearly, the condition (A3) can be verified using the data
(\ref{eqn:data_of_thm_boudary_data_to_spectral_data}). 
As 
\begin{equation*}
L_j f(y) = (f, \sum_{k \in I_j} (\pd_\nu \phi_k(y)) \eta \pd_\nu \phi_k )_{L^2(\pd M, d\mu)},
\quad f \in C_0^\infty(\Gamma_1), y \in \Gamma_2,
\end{equation*}
by varying $f \in C_0^\infty(\Gamma_1)$, we see that the map $L_j$ 
determines the function $l_j|_{\Gamma_1 \times \Gamma_2}$.

Let $x_0 \in \bar \Gamma_1 \cap \bar \Gamma_2$.
Given $l_0|_{\Gamma_1 \times \Gamma_2}$ and $\tilde \eta$, it is possible to compute 
\begin{equation*}
\pd_x^\alpha \pd_y^\beta \left( \frac{l_0(x,y)}{\tilde \eta(x)} \right), 
\quad (x, y) \in \Gamma_1 \times \Gamma_2, \quad \alpha, \beta \in \N^{n-1},
\end{equation*}
in any local coordinates $\Sigma$ taking $x_0$ to $0$.
By smoothness of $\eta$, $\tilde \eta$ and $\p_\nu \phi_0$, 
these functions are known also at $(x, y) = (0, 0)$.
Hence the condition (A1) can be verified using the data 
(\ref{eqn:data_of_thm_boudary_data_to_spectral_data}).

Taking $X = L^2(\Gamma_1, d\mu)$, $Y = L^2(\Gamma_2, d\mu)$ and 
$D = C_0^\infty(\Gamma_1)$ in the formulation of Remark \ref{lem:spans},
we see that the map $L_j$ determines the spaces $\tilde E_j^1$ and $E_j^2$.
Hence the condition (A2) can be verified using the data
(\ref{eqn:data_of_thm_boudary_data_to_spectral_data}), 
and the claim (i) is proved.

Let us show the claim (ii).
Let $x_0  \in \overline \Gamma_1 \cap \overline \Gamma_2$.
Lemma \ref{lem:unique_continuation} gives that, in local coordinates of $\p M$ taking $x_0$ to $0$, 
there is a multi-index $\alpha \in \N^{n-1}$ such that $\pd^\alpha \pd_\nu \phi_0(0) \ne 0$.
Hence the condition (A1) and Lemma \ref{lem:symmetry_for_derivatives}
imply, that $\pd^\beta (\eta \tilde \eta^{-1})(0) = 0$ for 
all nonzero multi-indices $\beta \in \N^{n-1}$.

Fix $j \in \N$ and, to simplify the notation, drop the subindices $j$ from now on.
By the condition (A2) 
\begin{equation*}
\tilde \eta e_l|_{\Gamma_1} = \sum_{k \in I} a_{lk} \eta \pd_\nu \phi_k|_{\Gamma_1},
\quad e_l|_{\Gamma_2} = \sum_{k \in I} b_{lk} \pd_\nu \phi_k|_{\Gamma_2}
\quad l \in I,
\end{equation*}
for some constant matrices $A := (a_{lk})_{l,k \in I}$ and $B := (b_{lk})_{l,k \in I}$. 

Fix $x_0 \in \overline \Gamma_1 \cap \overline \Gamma_2$, let $l \in I$ and 
define the function 
\begin{equation*}
\phi(p) := \sum_{k \in I} \left( a_{lk} \frac{\eta(x_0)}{\tilde \eta(x_0)} - b_{lk} \right)
  \phi_k(p), \quad p \in M.
\end{equation*}
We have seen that, in local coordinates of $\p M$ taking $x_0$ to $0$,
the equation $\pd^\beta (\eta \tilde \eta^{-1})(0) = 0$ holds for
all nonzero multi-indices $\beta \in \N^{n-1}$.
Hence for any multi-index $\alpha  \in \N^{n-1}$
\begin{eqnarray*}
0 &= \pd^\alpha e_l(0) - \pd^\alpha e_l(0)
\\&= 
 \pd_x^\alpha 
   \l. \left( \sum_{k \in I} a_{lk} \frac{\eta(x)}{\tilde \eta(x)} \pd_\nu \phi_k(x) \right) \r|_{x = 0} 
 - \pd_y^\alpha \l. \left( \sum_{k \in I}  b_{lk} \pd_\nu \phi_k(y) \right)\r|_{y = 0}
\\&= \sum_{k \in I} 
  \left( a_{lk} \frac{\eta(0)}{\tilde \eta(0)} - b_{lk} \right)
  \pd^\alpha \pd_\nu \phi_k(0)  
= \pd^\alpha \pd_\nu \phi(0).
\end{eqnarray*}
By Lemma \ref{lem:unique_continuation}, 
the coefficients 
\begin{equation*}
a_{lk} \frac{\eta(0)}{\tilde \eta(0)} - b_{lk}, 
\quad k,l \in I
\end{equation*}
vanish, and so $\eta(0) \tilde \eta(0)^{-1} A = B$.

Moreover
\begin{eqnarray*}
&(h, \tilde Lf)_{L^2(\pd M, dS_g)}
\\&\quad=  
\sum_{k \in I} (f, \tilde \eta e_k)_{L^2(\pd M, d\mu)} (h, e_k)_{L^2(\pd M, dS_g)}
\\&\quad= 
\sum_{k \in I} (f, \sum_{l \in I} a_{kl} \eta \pd_\nu \phi_l)_{L^2(\pd M, d\mu)} 
  (h, \sum_{m \in I} b_{km} \pd_\nu \phi_m)_{L^2(\pd M, dS_g)}
\\&\quad= \sum_{l,m \in I} \left( \sum_{k \in I} a_{kl} b_{km} \right) 
  (f, \pd_\nu \phi_l)_{L^2(\pd M, dS_g)} (h, \pd_\nu \phi_m)_{L^2(\pd M, dS_g)},
\end{eqnarray*}
for all $f \in C_0^\infty(\Gamma_1)$ and $h \in C_0^\infty(\Gamma_2)$.
On the other hand, the condition (A3) gives
\begin{eqnarray*}
(h, \tilde Lf)_{L^2(\pd M, dS_g)} 
&= (h, Lf)_{L^2(\pd M, dS_g)}
\\&= 
 \sum_{k \in I} (f,\pd_\nu \phi_k)_{L^2(\pd M, dS_g)} 
   (h, \pd_\nu \phi_k)_{L^2(\pd M, dS_g)},
\end{eqnarray*}
for all $f \in C_0^\infty(\Gamma_1)$ and $h \in C_0^\infty(\Gamma_2)$.

Denote $(\cdot, \cdot) := (\cdot, \cdot)_{L^2(\pd M, dS_g)}$.
By density of $C_0^\infty(\Gamma_p)$ in $L^2(\Gamma_p, dS_g)$, $p = 1, 2$,
\begin{equation*}
\sum_{l,m \in I} \left( \sum_{k \in I} a_{kl} b_{km} \right) 
  (f, \pd_\nu \phi_l) (h, \pd_\nu \phi_m)
 = \sum_{k \in I} (f,\pd_\nu \phi_k) 
   (h, \pd_\nu \phi_k),
\end{equation*}
for all $f \in L^2(\Gamma_1, dS_g)$ and $h \in L^2(\Gamma_2, dS_g)$.

Let $(f_l)_{l \in I}$ be biorthogonal with $(\pd_\nu \phi_l|_{\Gamma_1})_{l \in I}$ 
in $L^2(\Gamma_1, dS_g)$, and 
let$(h_m)_{m \in I}$ be biorthogonal with $(\pd_\nu \phi_m|_{\Gamma_2})_{m \in I}$
in $L^2(\Gamma_2, dS_g)$, that is,
\begin{equation*}
(f_{l'}, \p_\nu \phi_l) = \delta_{l'l}, \quad (h_{m'}, \p_\nu \phi_m) = \delta_{m'm}, 
\quad l', l, m, m' \in I.
\end{equation*}

Then
\begin{eqnarray*}
\sum_{k \in I} a_{kl'} b_{km'}
 &= \sum_{l,m \in I} \left( \sum_{k \in I} a_{kl} b_{km} \right) 
  (f_{l'}, \pd_\nu \phi_l) (h_{m'}, \pd_\nu \phi_m)
 \\&= \sum_{k \in I} (f_{l'},\pd_\nu \phi_k) 
   (h_{m'}, \pd_\nu \phi_k)
 = \delta_{l'm'}, \quad l',m' \in I.
\end{eqnarray*}

Denote $c := \eta(0)^{-1} \tilde \eta(0) > 0$. 
We have shown, that $I = A^T B = c B^T B$. Hence the matrix $\sqrt{c} B$ is orthogonal. 
To conclude, we observe that
\begin{equation*}
e_l|_{\Gamma_2} 
= \frac{1}{\sqrt{c}} \pd_\nu \sum_{k \in I} \sqrt{c} b_{lk} \phi_k|_{\Gamma_2},
\quad l \in I,
\end{equation*}
where $(\sum_{k \in I} \sqrt{c} b_{lk} \phi_k)_{l \in I}$ 
is an orthonormal basis of eigenfunctions corresponding the eigenvalue $\lambda_j$.
\end{proof}

As discussed in the previous section, the operator $\Lambda_{\Gamma_1, \Gamma_2}$
determines the collection (\ref{eqn:data_of_thm_boudary_data_to_spectral_data}).
So by previous theorem, 
if $\overline \Gamma_1 \cap \overline \Gamma_2 \ne \emptyset$,
then the operator $\Lambda_{\Gamma_1, \Gamma_2}$
determines the collection (\ref{eq:boundary_spectral_data}).
The collection (\ref{eq:boundary_spectral_data}) determines 
the manifold up to an isometry by \cite[Chapter 4.4]{KKL}. 
This proves Theorem \ref{thm:touching_main}.

\section{Inverse problem with observations far away from sources}

In this section we prove Theorems \ref{thm:combining_data_infinite_time}
and \ref{thm:combining_data_finite_time}. 

\begin{proof}[Proof of Theorem \ref{thm:combining_data_infinite_time}]
Denote $L^{p \to q}_j = L_{\Gamma_p, \Gamma_q; j}$.
It is enough to show, that the collection 
\begin{equation}
\label{eqn:data_of_thm_combining_data_infinite_time}
\{ (\lambda_j, L_j^{1 \to 2}, L_j^{1 \to 2}, L_j^{2 \to 3})
\ |\ j \in \N \}
\end{equation}
determines a collection of type (\ref{eq:boundary_spectral_data}).

Denote $\Sigma = \Gamma_1 \cup \Gamma_2 \cup \Gamma_3$, and 
choose a smooth positive measure $d\mu$ on $\Sigma$. 
There is a positive function $\eta \in C^\infty(\Sigma)$ such that $\eta d\mu = dS_g|_{\Sigma}$.
Define for all $j \in \N$
\begin{eqnarray*}
\tilde E_j^p &:= \linspan(\eta \pd_\nu \phi_k|_{\Gamma_p})_{k \in I_j}, 
\quad p = 1, 2,
\\E_j^q &:= \linspan(\pd_\nu \phi_k|_{\Gamma_q})_{k \in I_j},
\quad q = 2, 3.
\end{eqnarray*}

Choose a positive function $\tilde \eta \in C^\infty(\Sigma)$ and 
real-valued functions $e_k \in C^\infty(\Sigma)$, $k \in \N$, such that
the following two conditions hold:
\begin{itemize}
\item[(B1)] for all $j \in \N$
\begin{eqnarray*}
\linspan(\tilde \eta e_k|_{\Gamma_p})_{k \in I_j} &= \tilde E_j^p,
\quad p = 1, 2,
\\\linspan(e_k|_{\Gamma_q})_{k \in I_j} &= E_j^q,
\quad q = 2, 3.
\end{eqnarray*}
\item[(B2)] $\tilde L_j^{p \to q} = L_j^{p \to q}$ for all $j \in \N$ and
$(p,q) = (1,2), (1,3), (2,3)$, where
\begin{equation*}
\tilde L_j^{p \to q} f(y) := \sum_{k \in I_j} (f, \tilde \eta e_k)_{L^2(\pd M, d\mu)} e_k(y)
\quad f \in C_0^\infty(\Gamma_p), y \in \Gamma_q.
\end{equation*}
\end{itemize}

Again, such functions $\tilde \eta$ and $e_k$ exist, as 
$\tilde \eta = \eta$ and $e_k = \pd_\nu \phi_k|_\Sigma$ satisfy the conditions (B1) and (B2).
It is enough to show the following two statements:
\begin{itemize}
\item[(i)] 
We can verify using the data (\ref{eqn:data_of_thm_combining_data_infinite_time}), whether
any given functions $\tilde \eta \in C^\infty(\Sigma)$ and $e_k \in C^\infty(\Sigma)$, $k \in \N$,
satisfy the conditions (B1) and (B2).
\item[(ii)] There is an orthonormal basis of eigenfunctions $(\psi_k)_{k \in \N}$ 
of operator $A$ and a constant $C > 0$, not depending on $k$, such that
\begin{equation*}
e_k|_{\Gamma_2} = C \pd_\nu \psi_k|_{\Gamma_2}, \quad k \in \N.
\end{equation*}
\end{itemize}

Analogously with the proof of Theorem \ref{thm:boudary_data_to_spectral_data},
the maps $L^{1 \to 2}_j$, $L^{1 \to 3}_j$ and $L^{2 \to 3}_j$ determine
the spaces $\tilde E_j^1$, $E_j^2$, $E_j^3$ and $\tilde E_j^2$.
Hence the claim (i) is proved.

Let us show the claim (ii). 
Fix $j \in \N$ and, to simplify the notation, drop the subindices $j$ from now on.
The condition (B1) gives, that for all $l \in I$
\begin{eqnarray*}
&\tilde \eta e_l|_{\Gamma_1} = \sum_{k \in I} a_{lk} \eta \pd_\nu \phi_k|_{\Gamma_1},
\quad \tilde \eta e_l|_{\Gamma_2} = \sum_{k \in I} \tilde b_{lk} \eta \pd_\nu \phi_k|_{\Gamma_2},
\\&e_l|_{\Gamma_2} = \sum_{k \in I} b_{lk} \pd_\nu \phi_k|_{\Gamma_2},
\quad e_l|_{\Gamma_3} = \sum_{k \in I} c_{lk} \pd_\nu \phi_k|_{\Gamma_3},
\end{eqnarray*}
for some constant matrices 
\begin{equation*}
A := (a_{lk})_{l,k \in I}, \quad \tilde B := (\tilde b_{lk})_{l,k \in I}, 
\quad B := (b_{lk})_{l,k \in I}, \quad C := (c_{lk})_{l,k \in I}.
\end{equation*}

For the smallest eigenvalue $\lambda_0$, the space of eigenfunctions is one dimensional,
and so 
\begin{equation*}
0 = e_0(y) - e_0(y) 
= \left( \frac{\eta(y)}{\tilde \eta(y)} \tilde b_{00} - b_{00} \right) \pd_\nu \phi_0(y)
\end{equation*}
for all $y \in \Gamma_2$.
By Lemma \ref{lem:unique_continuation}, the set 
\begin{equation*}
N := \{ y \in \Gamma_2 : \pd_\nu \phi_0(y) = 0 \}
\end{equation*}
does not contain a nonempty open set of $\p M$.
Hence $\bar{\Gamma_2 \setminus N} = \bar \Gamma_2$.
Moreover, $\eta \tilde \eta^{-1} \tilde b_{00} - b_{00} = 0$ in 
$\Gamma_2 \setminus N$ and by continuity also in the whole set $\Gamma_2$.

Denote by $c$ the constant $\eta^{-1} \tilde \eta|_{\Gamma_2} > 0$.
As $(\pd_\nu \phi_k|_{\Gamma_2})_{k \in I}$ are linearly independent,
$\tilde B = cB$ for all $j \in \N$. 
Moreover, we may use the condition (B2) as we used the corresponding condition in
the proof of Theorem \ref{thm:boudary_data_to_spectral_data}, and get
\begin{equation*}
A^T B = I, \quad A^T C = I, \quad \tilde B^T C = I.
\end{equation*}
Hence $B = C$ and $cB^T B = I$.

To conclude, we observe that 
\begin{equation*}
e_l|_{\Gamma_2} 
= \frac{1}{\sqrt{c}} \pd_\nu \sum_{k \in I} \sqrt{c} b_{lk} \phi_k|_{\Gamma_2},
\quad l \in I,
\end{equation*}
where $(\sum_{k \in I} \sqrt{c} b_{lk} \phi_k)_{l \in I}$ 
is an orthonormal basis of eigenfunctions corresponding the eigenvalue $\lambda_j$.
\end{proof}

We prove Theorem \ref{thm:combining_data_finite_time} by 
reduction to Theorem \ref{thm:combining_data_infinite_time} 
using a time continuation argument similar to \cite{KuLa2}.

\begin{lemma}
Suppose that $\Gamma_1, \Gamma_2 \subset \p M$ are open and nonempty.
Denote
\begin{equation*}
T^* := 2 \max \{ d(x,y) : x \in \Gamma_1, y \in M \}.
\end{equation*}
If $T^* < t_0 < T$, 
then the smooth manifolds $\Gamma_1$, $\Gamma_2$, the operator $\Lambda_{\Gamma_1, \Gamma_2}^T$ 
and the inner products 
\begin{equation}
\label{eqn:continuation_inner_products}
(u^f(t_0), u^h(t_0))_{L^2(M)}, \quad f, h \in C_0^\infty(\Gamma_1 \times (0,T))
\end{equation}
determine the operator $\Lambda_{\Gamma_1, \Gamma_2}^{T+\delta}$
for $\delta < t_0 - T^*$.
\label{lem:continuation_of_data}
\end{lemma}

\begin{proof}
Denote by $Y_s$ the time delay operator 
\begin{equation*}
Y_s f (\cdot, t) := f(\cdot, t - s), \quad t,s \in \R.
\end{equation*}
As the coefficients of the wave equation (\ref{def:wave_eq}) 
are time-independent,
$v^{Y_s f}(x, t) = (Y_s v^f)(x, t)$ and $(\Lambda Y_s f)(x, t) = (Y_s \Lambda f)(x, t)$.

Let $f\in C_0^\infty(\Gamma_1 \times (0, T + \delta)$ and
choose $h \in C_0^\infty(\Gamma_1 \times (0, t_0))$
and $h' \in C_0^\infty(\Gamma_1 \times (\delta, T + \delta))$
such that $f = h + h'$.

Let $\epsilon \in (0, \delta)$. 
As $\supp(Y_{-\delta} h') \subset \Gamma_1 \times (0, T)$,
the operator $\Lambda_{\Gamma_1, \Gamma_2}^T$ determine the function
\begin{equation*}
\Lambda h'(\cdot, T + \epsilon) 
= (Y_{-\delta} \Lambda h')(\cdot, T - (\delta - \epsilon))
= (\Lambda_{\Gamma_1, \Gamma_2}^T Y_{-\delta}h')(\cdot, T - (\delta - \epsilon)), 
\end{equation*}
in $\Gamma_2$.
Therefore, it is enough to show that the given data determine also 
\begin{equation*}
\Lambda h(x, T + \epsilon), \quad x \in \Gamma_2.
\end{equation*}

Consider a sequence 
$(h_j)_{j=1}^\infty \subset C_0^\infty(\Gamma_1 \times (0,t_0 - \delta))$ 
satisfying the following two conditions.
\begin{itemize}
\item[(C1)] $\lim_{j \to \infty} v^{Y_\delta h_j}(t_0) = v^h(t_0)$ in $H^1(M)$,
\item[(C2)] $\lim_{j \to \infty} \pd_t v^{Y_\delta h_j}(t_0) = \pd_t v^h(t_0)$ 
    in $L^2(M)$.
\end{itemize}
Such a sequence exists, since $t_0 > T^*$, and thus we see exactly as in 
\cite[Thm. 4.28]{KKL} that the set
\begin{equation*}
\{(v^f(t_0 - \delta), \pd_t v^f(t_0 - \delta)) : 
    f \in C_0^\infty(\Gamma_1 \times (0,t_0 - \delta)) \}
\end{equation*}
is dense in $H_0^1(M) \times L^2(M)$.

Let us prove, that (C1) is equivalent with
\begin{itemize}
\item[(C1')] For all $c > 0$ 
\begin{equation*}
\lim_{j \to \infty} \left( ( -\Delta_g w_j + q w_j, w_j)_{L^2(M)} + c(w_j, w_j)_{L^2(M)} \right)= 0,
\end{equation*}
where $w_j := v^{Y_\delta h_j}(t_0) - v^{h}(t_0)$.
\end{itemize}

As $\supp(Y_\delta h_j) \subset \Gamma_1 \times (0,t_0)$ and $\supp(h) \subset \Gamma_1 \times (0,t_0)$,
we have that $w_j|_{\pd M} = 0$. Hence 
\begin{equation*}
-(\Delta_g w_j, w_j)_{L^2(M)} = (dw_j, dw_j)_{L^2(M)},
\end{equation*}
where $d$ is the exterior derivative on $M$.
If (C1) holds, then 
\begin{eqnarray*}
&|( -\Delta_g w_j + q w_j, w_j)_{L^2(M)} + c(w_j, w_j)_{L^2(M)}|
\\&\quad\le \norm{dw_j}_{L^2(M)}^2 + \norm{q + c}_{L^\infty(M)} \norm{w_j}_{L^2(M)}^2
\\&\quad\to 0, \quad \text{as $j \to \infty$}.
\end{eqnarray*}
For large enough $c > 0$ there is a constant $c_0 > 0$ 
such that $q + c \ge c_0$. Hence if (C1') holds, then
\begin{equation*}
\norm{dw_j}_{L^2(M)}^2 + c_0 \norm{w_j}_{L^2(M)}^2 \le 
( -\Delta_g w_j + q w_j, w_j)_{L^2(M)} + c(w_j, w_j)_{L^2(M)} \to 0,
\end{equation*}
as $j \to \infty$, and (C1) holds.
Therefore (C1) and (C1') are equivalent.

Next we observe that
\begin{eqnarray*}
&-\pd_s^2 (v^{Y_s(Y_\delta h_j - h)}(t_0), v^{Y_\delta h_j - h}(t_0))_{L^2(M)}|_{s = 0} 
\\&\quad= (-\pd_t^2 w_j, w_j)_{L^2(M)} 
= ( (-\Delta_g + q) w_j, w_j)_{L^2(M)}.
\end{eqnarray*}
Hence the condition (C1') can be verified 
for given functions $(h_j)_{j=1}^\infty$ and $h$
using the inner products 
(\ref{eqn:continuation_inner_products}).

Similarly,
\begin{equation*}
\pd_{s_1} \pd_{s_2} 
    (v^{Y_{s_1}(Y_\delta h_j - h)}(t_0), v^{Y_{s_2}(Y_\delta h_j - h)}(t_0))_{L^2(M)}
    |_{s_1 = 0, s_2 = 0}
= (\pd_t w_j, \pd_t w_j)_{L^2(M)},
\end{equation*}
and condition (C2) can be verified 
for given functions $(h_j)_{j=1}^\infty$ and $h$
using the inner products 
(\ref{eqn:continuation_inner_products}).

As $v^{Y_\delta h_j} - v^{h} = 0$ on $\pd M \times [t_0, \infty)$,
conditions (C1) and (C2) together with 
the continuous dependency of the solution of the wave equation on the initial data, see e.g.
\cite[Thm. 2.30]{KKL}, give
\begin{equation*}
\lim_{j \to \infty} \pd_\nu (v^{Y_\delta h_j} - v^{h}) = 0,
\quad \text{in $L^2(\pd M \times (t_0, \infty))$.}
\end{equation*}

We have seen that, the inner products (\ref{eqn:continuation_inner_products}) determine
for any $h \in C_0^\infty(\Gamma_1 \times (0, t_0))$
the nonempty set 
\begin{equation*}
\{ (h_j)_{j=1}^\infty \subset C_0^\infty(\Gamma_1 \times (0,t_0 - \delta))\ |\ 
\text{(C1) and (C2) hold} \},
\end{equation*}
and that any sequence in this set satisfies
\begin{equation*}
\Lambda h(x, T + t) 
= \lim_{j \to \infty} (\Lambda Y_\delta h_j)(x, T + t)
= \lim_{j \to \infty} \Lambda_{\Gamma_1, \Gamma_2}^T h_j(x, T - (\delta - t))
\end{equation*}
in $L^2(\Gamma_2 \times (T, T + \delta))$.
As $\Lambda h \in C^\infty(\Gamma_2 \times (0, T + \delta))$, 
the inner products (\ref{eqn:continuation_inner_products}) and the operator $\Lambda_{\Gamma_1, \Gamma_2}^T$ determine $\Lambda h(x, T + \epsilon)$ pointwise for $x \in \Gamma_2$ and $\epsilon \in (0, \delta)$.
\end{proof}

Next we prove the last of the three main theorems formulated in the introduction.

\begin{proof}[Proof of Theorem \ref{thm:combining_data_finite_time}]
By Lemma \ref{lem:symmetrization_of_data} the operators 
$\Lambda_{\Gamma_1, \Gamma_2}^T, \Lambda_{\Gamma_1, \Gamma_3}^T, \Lambda_{\Gamma_2, \Gamma_3}^T$
determine the operators
\begin{equation}
\label{eqn:ops_of_finite_time_thm}
\Lambda_{\Gamma_p, \Gamma_q}^T, \quad p,q = 1, 2, 3,\ p \ne q,
\end{equation}

We use the time delay operator $Y_s$ defined in the proof of Lemma 
\ref{lem:continuation_of_data}. Define
\begin{equation*}
B[f,h] := \int_0^T \int_{\pd M} 
\left( \pd_\nu v^f \bar{v^h} - v^f \bar{\pd_\nu v^h} \right) dS_g dt,
\quad f,h \in C_0^\infty(\pd M \times (0, \infty)),
\end{equation*}
and let $t_0 := T/2$.
We recall Blagovestchenskii identity \cite[Lem. 4.16]{KKL},
originating from \cite{Bl}, 
\begin{equation*}
(v^f(t_0), v^h(t_0))_{L^2(M)} = 
\frac{1}{2} \int_{-t_0}^{t_0} (\sign\ s) B[Y_{t_0+s}f, Y_{t_0-s}h] ds,
\end{equation*}
where $f,h \in C_0^\infty( \pd M \times (0,T))$.
By this identity, the operators (\ref{eqn:ops_of_finite_time_thm})
determine the inner products 
\begin{equation}
\label{eqn:inner_products_t0}
(v^f(t_0), v^h(t_0))_{L^2(M)}, 
\quad f \in C_0^\infty(\Gamma_p \times (0,T)),
\ h \in C_0^\infty(\Gamma_q \times (0,T)),
\end{equation}
for $p,q = 1, 2, 3$, $p \ne q$.

Next we will show that the operators (\ref{eqn:ops_of_finite_time_thm})
determine the inner products (\ref{eqn:inner_products_t0}) also for $p=q$, $p = 1, 2, 3$.

Let $f \in C_0^\infty(\Gamma_2 \times (0,T))$ and
consider a sequence $(\tilde{f_j})_{j=1}^\infty \subset C_0^\infty(\Gamma_3 \times (0, T))$
satisfying the following two conditions.
\begin{itemize}
\item[(D1)] For all $h \in C_0^\infty(\Gamma_1 \times (0, T))$
\begin{equation*}
\lim_{j \to \infty} (v^f(t_0) - v^{\tilde{f_j}}(t_0), v^h(t_0))_{L^2(M)} = 0.
\end{equation*}
\item[(D2)] The sequence $(\tilde{f_j})_{j=1}^\infty$ is bounded in 
$L^2(\Gamma_3 \times (0,T))$.
\end{itemize}

By assumption (A), there is $\tilde{f} \in L^2(\Gamma_3 \times (0, T))$
such that $v^f(t_0) = v^{\tilde{f}}(t_0)$.
Thus, there is a sequence $(\tilde{f_j})_{j=1}^\infty \subset C_0^\infty(\Gamma_3 \times (0, T))$
such that 
\begin{equation*}
\lim_{j \to \infty} \tilde{f_j} = \tilde{f}, \quad \text{in $L^2(\Gamma_3 \times (0, T))$}.
\end{equation*}
By \cite[Lem. 2.42]{KKL},
$v^{\tilde{f_j}}(t_0) \to v^{\tilde{f}}(t_0)$ in $L^2(M)$ as $j \to \infty$.
Hence a sequence satisfying the conditions (D1) and (D2) exists.

Let us next show that 
\begin{equation}
\label{eqn:inner_product_is_determined}
\norm{v^f(t_0)}_{L^2(M)}^2 = \lim_{j \to \infty} (v^{\tilde{f_j}}(t_0), v^f(t_0))_{L^2(M)}.
\end{equation}
As $t_0 > 2 d(x, y)$ for all $x \in \Gamma_1$ and $y \in M$, \cite[Thm. 3.10]{KKL} gives that the set
\begin{equation*}
\{v^h(t_0)\ |\ h \in C_0^\infty(\Gamma_1 \times (0,T)) \}
\end{equation*}
is dense in $L^2(M)$.
Let $\epsilon > 0$ and choose $h \in C_0^\infty(\Gamma_1 \times (0,T))$ such that
\begin{equation*}
\norm{v^f(t_0) - v^h(t_0)}_{L^2(M)} < \epsilon. 
\end{equation*}

By \cite[Lem. 2.42]{KKL} there is $C > 0$, 
and by the condition (D1) there is $J \in \N$ such that for $j \ge J$
\begin{eqnarray*}
&\left|\left(v^f(t_0) - v^{\tilde{f_j}}(t_0), v^f(t_0)\right)_{L^2(M)}\right|
\\&\quad\le \left|\left(v^f(t_0) - v^{\tilde{f_j}}(t_0), v^f(t_0) - v^h(t_0)\right)_{L^2(M)}\right|
\\&\quad\quad+ \left|\left(v^f(t_0) - v^{\tilde{f_j}}(t_0), v^h(t_0)\right)_{L^2(M)}\right|
\\&\quad\le C\norm{f - \tilde f_j}_{L^2(\pd M \times (0,T))} \epsilon + \epsilon.
\end{eqnarray*}
By the condition (D2) 
\begin{equation*}
\sup_{j \in \N} \norm{f - \tilde f_j}_{L^2(\pd M \times (0,T))} < \infty.
\end{equation*}
Hence the equation (\ref{eqn:inner_product_is_determined}) is valid.

By \cite[Thm. 3.10]{KKL} the functions $v^f(t_0)$, 
$f \in C_0^\infty(\Gamma_2 \times (0,T))$, are dense in $L^2(M)$.
Hence
\begin{equation}
\label{eqn:inner_product_is_determined2}
\norm{v^h(t_0)}_{L^2(M)} = \sup (v^h(t_0), v^f(t_0))_{L^2(M)} 
\end{equation}
where $h \in C_0^\infty(\Gamma_p \times (0,T))$, $p = 1, 3$,
and the supremum is taken over all $f \in C_0^\infty(\Gamma_2 \times (0,T))$ such that
$\norm{v^f(t_0)}_{L^2(M)} = 1$.

The condition (D1) can be verified 
for any $f$ and $(f_j)_{j=1}^\infty$ using the 
inner products (\ref{eqn:inner_products_t0}) for $p=2, 3$, $q = 1$.
Therefore, these inner products determine
for any $f \in C_0^\infty(\Gamma_2 \times (0,T))$ the nonempty set 
\begin{equation*}
\{ (\tilde{f_j})_{j=1}^\infty \subset C_0^\infty(\Gamma_3 \times (0, T))\ |\ \text{(D1), (D2) hold} \}.
\end{equation*}
By equation (\ref{eqn:inner_product_is_determined}) any sequence in this set 
together with inner products (\ref{eqn:inner_products_t0}) for $p=3$ and $q=2$
determine $\norm{v^f(t_0)}_{L^2(M)}$.

As $f \in C_0^\infty(\Gamma_2 \times (0,T))$ can be chosen freely,
the inner products (\ref{eqn:inner_products_t0}) for $p=2, 3$, $q = 1$ and for $p=3$, $q=2$
together with  polarization identity determine the inner products (\ref{eqn:inner_products_t0})
for $p = q = 2$. 

The equation (\ref{eqn:inner_product_is_determined2}), polarization identity and
the inner products (\ref{eqn:inner_products_t0})
for $p=1,2,3$, $q = 2$ 
determine the inner products (\ref{eqn:inner_products_t0}) for
$p = q = 1, 3$.

Therefore, the operators (\ref{eqn:ops_of_finite_time_thm})
determine the inner products 
\begin{equation}
\label{eqn:all_inner products}
(v^f(t_0), v^h(t_0))_{L^2(M)}, 
\quad f, h \in C_0^\infty(\Gamma_p \times (0,T)),\ p=1,2,3.
\end{equation}

Choose $\delta \in (0, t_0 - T^*)$,
where $T^*$ is defined as in Lemma \ref{lem:continuation_of_data}.
By Lemma \ref{lem:continuation_of_data} 
the operators (\ref{eqn:ops_of_finite_time_thm}) and the inner products (\ref{eqn:all_inner products})
determine the operators
\begin{equation}
\label{eqn:time_continued_ops}
\Lambda_{\Gamma_p, \Gamma_q}^{T + \delta}, \quad p,q = 1, 2, 3,\ p \ne q.
\end{equation}


Repeating this construction, we see that the operators
\begin{equation*}
\Lambda_{\Gamma_p, \Gamma_q}^{T + m\delta}, \quad p,q = 1, 2, 3,\ p \ne q,
\end{equation*}
are determined for all $m \in \N$. 
The claim follows from Theorem \ref{thm:combining_data_infinite_time}.
\end{proof}

\appendix
\section*{Appendix}

Next we prove the Lemma \ref{lem:symmetrization_of_data} stating that the operator 
$\Lambda_{\Gamma_1, \Gamma_2}^T$ determines the operator
$\Lambda_{\Gamma_2, \Gamma_1}^T$. 

\begin{proof}[Proof of Lemma \ref{lem:symmetrization_of_data}]
Define $u := R v^{R f}$, where $v^{Rf}$ is the solution of the equation (\ref{def:wave_eq})
with the boundary data $Rf \in C_0^\infty(\Gamma_1 \times (0,T))$.
Then $u(x,t) = v^{R f}(x, T - t)$ satisfies the equation
\begin{eqnarray*}
&\p_t^2 u + a(x,D)u = 0,\ \text{in $M \times (0, T)$},
\\& u|_{\p M \times (0,T)} = f,
\\& u|_{t = T} = \p_t u|_{t = T} = 0.
\end{eqnarray*}

Integration by parts gives
\begin{eqnarray*}
&(f, \Lambda_{\Gamma_2, \Gamma_1}^T h)_{L^2(\p M \times (0,T))} 
- (R \Lambda_{\Gamma_1, \Gamma_2}^T R f, h)_{L^2(\p M \times (0,T))}
\\&\quad= \int_0^T \int_{\p M} 
\left( u(x,t) \p_\nu v^h(x,t) - (\p_\nu v^{Rf})(x, T - t) v^h(x,t) \right) dS_g(x) dt
\\&\quad= \int_0^T \int_M 
\left( u(x,t) \Delta_g v^h(x,t) - (\Delta_g u)(x, t) v^h(x,t) \right) dV_g(x) dt
\\&\quad= \int_0^T \int_M 
\left( u(x,t) \p_t^2 v^h(x,t) - (\p_t^2 u)(x, t) v^h(x,t) \right) dV_g(x) dt
\\&\quad=
\left[ \int_M \left( u(x,t) \p_t v^h(x,t) - (\p_t u)(x, t) v^h(x,t) \right) dV_g(x) \right]_{t=0}^{t=T} = 0,
\end{eqnarray*}
since $u|_{t = T} = \p_t u|_{t = T} = 0$ and $v^f|_{t = 0} = \p_t v^f|_{t = 0} = 0$.
\end{proof}

\noindent
{\bf Acknowledgements.}
The authors were partly supported by Finnish Centre of Excellence in Inverse Problems Research,
Academy of Finland COE 213476.
LO was partly supported also by Finnish Graduate School in Computational Sciences.

\section*{References}

\end{document}